\documentclass[11pt]{amsart}

\usepackage{amsmath}
\usepackage{amsxtra}
\usepackage{amscd}
\usepackage{amsthm}
\usepackage{amsfonts}
\usepackage{amssymb}
\usepackage{mathtools}
\usepackage{eucal}
\usepackage[all]{xy}
\usepackage{graphicx}
\usepackage{tikz-cd}
\usepackage{mathrsfs}
\usepackage{euler}
\usepackage{color}
\usepackage{longtable}
\usepackage{caption}
\usepackage{enumerate}
\usepackage{faktor}
\usepackage{bbm}
\usepackage{stmaryrd}
\usepackage{rotating}
\usepackage[T1]{fontenc} 
\usepackage[utf8]{inputenc} 
\usepackage{tikz}
\usepackage{tikz-cd}
\usetikzlibrary{matrix,arrows,decorations.pathmorphing}
\usepackage{hyperref}
\usepackage{float}

\usepackage[bottom=1.25in, top=1.5in, right=1.5in, left=1.5in]{geometry}

\theoremstyle{plain}
  \newtheorem{theorem}{Theorem}
	
  \newtheorem{proposition}[theorem]{Proposition}
	
  \newtheorem{lemma}[theorem]{Lemma}

  \newtheorem{corollary}[theorem]{Corollary}

\theoremstyle{definition}

	\newtheorem{conj}[theorem]{Conjecture}
	
\theoremstyle{remark}
	
	\newtheorem{remark}[theorem]{Remark}

\DeclareMathAlphabet{\mathcal}{OMS}{cmsy}{m}{n}

\newcommand{\one}{\mathbbm{1}}
\newcommand{\A}{\mathbb{A}}

\newcommand{\C}{\mathbb{C}}

\newcommand{\F}{\mathbb{F}}
\newcommand{\G}{\mathbb{G}}

\newcommand{\K}{K}
\newcommand{\N}{\mathbb{N}}

\newcommand{\Q}{\mathbb{Q}}
\newcommand{\R}{\mathbb{R}}

\newcommand{\Z}{\mathbb{Z}}

\newcommand{\sce}{\mathscr{E}}

\newcommand{\cala}{\mathcal{A}}

\newcommand{\cald}{\mathcal{D}}

\newcommand{\calh}{\mathcal{H}}

\newcommand{\calm}{\mathcal{M}}
\newcommand{\calo}{\mathcal{O}}
\newcommand{\calp}{\mathcal{P}}

\newcommand{\cals}{\mathcal{S}}

\newcommand{\calu}{\mathcal{U}}

\newcommand{\mfa}{\mathfrak{a}}
\newcommand{\mfb}{\mathfrak{b}}
\newcommand{\mfc}{\mathfrak{c}}
\newcommand{\mff}{\mathfrak{f}}

\newcommand{\mfl}{\mathfrak{l}}

\newcommand{\mfs}{\mathfrak{S}}
\newcommand{\rmn}{\mathrm{N}}

\newcommand{\Ad}{{\mathrm{Ad}}}

\newcommand{\Aut}{{\mathrm{Aut}}}

\renewcommand{\Im}{{\mathrm{Im}}}

\newcommand{\gal}{\mathrm{Gal}}
\newcommand{\Gal}{{\mathrm{Gal}}}

\newcommand{\Hom}{{\mathrm{Hom}}}

\newcommand{\ind}{\mathrm{Ind}}
\newcommand{\Ind}{{\mathrm{Ind}}}

\newcommand{\Reg}{{\mathrm{Reg}}}

\newcommand{\gl}{\mathrm{GL}}
\newcommand{\GL}{{\mathrm{GL}}}

\newcommand{\GO}{{\mathrm{GO}}}
\newcommand{\GSO}{{\mathrm{GSO}}}

\newcommand{\PGL}{{\mathrm{PGL}}}

\newcommand{\PSL}{{\mathrm{PSL}}}

\newcommand{\SO}{{\mathrm{SO}}}
\newcommand{\SL}{{\mathrm{SL}}}

\newcommand{\Frob}{{\mathrm{Frob}}}
\newcommand{\frob}{{\mathrm{Frob}}}
\newcommand{\Id}{{\mathrm{Id}}}

\newcommand{\SD}{{\mathrm{SD}}}
\newcommand{\Tr}{{\mathrm{Tr}}}

\newcommand{\hv}{\mathrm{HV}}

\newcommand{\rs}{\mathrm{RS}}

\newcommand{\opt}{{\mathrm{opt}}}
\newcommand{\stark}{{\mathrm{Stark}}}
\newcommand{\Stark}{{\mathrm{Stark}}}

\newcommand{\bs}{\backslash}
\newcommand{\wt}{\widetilde}

\newcommand{\lra}{{\, \longrightarrow \,}}
\newcommand{\iso}{\, \xrightarrow{\widesim{}} \,}

\newcommand{\paren}[1]{\mathopen{}\left(#1\right)\mathclose{}}

\newcommand{\sbrac}[1]{\mathopen{}\left[#1\right]\mathclose{}}
\newcommand{\abrac}[1]{\mathopen{}\left\langle#1\right\rangle\mathclose{}}
\newcommand{\verts}[1]{\mathopen{}\left\lvert#1\right\rvert\mathclose{}}
\newcommand{\norm}[1]{\mathopen{}\left\lvert\left\lvert#1\right\rvert\right\rvert\mathclose{}}
\newcommand\restr[2]{{
  \left.\kern-\nulldelimiterspace 
  #1 
  \right|_{#2} 
  }}

\newcommand{\pair}[1]{\abrac{#1}}
\newcommand{\wh}{\widehat}

\newcommand{\widesim}[2][2]{
  \mathrel{\overset{#2}{\scalebox{#1}[1]{$\sim$}}}
}

\makeatletter
\AtBeginDocument
 {
    \def\@thm#1#2#3{%
      \ifhmode
        \unskip\unskip\par
      \fi
      \normalfont
      \trivlist
      \let\thmheadnl\relax
      \let\thm@swap\@gobble
      \let\thm@indent\indent 
      \thm@headfont{\scshape}
      \thm@notefont{\fontseries\mddefault\upshape}%
      \thm@headpunct{.}
      \thm@headsep 5\p@ plus\p@ minus\p@\relax
      \thm@space@setup
      #1
      \@topsep \thm@preskip               
      \@topsepadd \thm@postskip           
      \def\dth@counter{#2}%
      \ifx\@empty\dth@counter
        \def\@tempa{%
          \@oparg{\@begintheorem{#3}{}}[]%
        }%
      \else
        \H@refstepcounter{#2}%
        \hyper@makecurrent{#2}%
        \let\Hy@dth@currentHref\@currentHref
        \AddToHookNext{para/begin}{\MakeLinkTarget*{\Hy@dth@currentHref}}%
        \def\@tempa{%
          \@oparg{\@begintheorem{#3}{\csname the#2\endcsname}}[]%
        }%
      \fi
      \@tempa
    }%
  \dth@everypar={%
    \@minipagefalse \global\@newlistfalse
    \@noparitemfalse
    \if@inlabel
      \global\@inlabelfalse
      \begingroup \setbox\z@\lastbox
       \ifvoid\z@ \kern-\itemindent \fi
      \endgroup
      \unhbox\@labels
    \fi
    \if@nobreak \@nobreakfalse \clubpenalty\@M
    \else \clubpenalty\@clubpenalty \everypar{}%
    \fi
  }%
}

\title[The Harris--Venkatesh conjecture for derived Hecke operators II]
{The Harris--Venkatesh conjecture for derived Hecke operators II:
a unified Stark conjecture}

\author{Robin Zhang}
\address{Department of Mathematics, Columbia University \newline
	\indent Department of Mathematics, Massachusetts Institute of Technology}
\email{rzhang@math.columbia.edu, robinz@mit.edu}

\date{June 3, 2024}

\begin{document}

\begin{abstract}
	\normalsize
	We prove a compatibility
	theorem between the Stark conjecture and the
	Harris--Venkatesh conjecture for
	imaginary dihedral modular forms of weight $1$.
	The key technical input is a general
	two-variable $\PGL_2$ Siegel--Weil formula
	that precisely gives the Laurent series coefficients
	of Eisenstein series as a dual theta lift of
	Eisenstein series.
	This two-variable Siegel--Weil formula is applied
	to Rankin--Selberg periods of imaginary dihedral optimal forms
	and a refinement of Stark's formula relating $L$-values
	with elliptic units, yielding compatibility between
	derived Hecke operators and
	adjoint $L$-values for Deligne--Serre representations.
	We formulate a general unified conjecture
	encompassing the Stark and Harris--Venkatesh conjectures
	and conceptually reinterpret the action of derived Hecke operators
	as modular Stark unit data at almost all primes.
\end{abstract}

\maketitle

\setcounter{tocdepth}{1}
\tableofcontents

\section{Introduction}
\label{sec:intro}

The Stark conjecture \cite{stark-1971,stark-1975,stark-1976,stark-1980}
relates the leading
coefficient of Artin $L$-functions to the existence of
certain algebraic units called \textit{Stark units}.
The work of Harris--Venkatesh \cite{hv} examines
modular forms of weight-$1$, which occur in both
the degree-$0$ and degree-$1$ coherent cohomology of the
same weight $1$ automorphic vector bundle unlike modular forms
of weight $\geq 2$, and construct
\textit{derived Hecke operators}
$T_{p, N}: H^0(X_1(N)_{\Z/(p-1)\Z}, \omega) \rightarrow
H^1(X_1(N)_{\Z/(p-1)\Z}, \omega)$.
These degree-shifting operators are the subject of
the vast conjectures of
Venkatesh \cite{venkatesh-2},
Prasanna--Venkatesh \cite{prasanna-venkatesh},
and Galatius--Venkatesh \cite{galatius-venkatesh} on the interplay
between derived Hecke algebras, rational structures of
motivic cohomology groups, and derived Galois deformation rings.
The Harris--Venkatesh conjecture \cite{hv} relates the action of
derived Hecke operators on modular forms of weight $1$
to algebraic units modulo $p$ for almost
all primes $p$.
Inspired by a question of Gross,
we prove that the Harris--Venkatesh conjecture and the
Stark conjecture actually describe the same units in the
case of imaginary dihedral modular forms of weight $1$
and suggest that there is a more general phenomenon.

Let $f = \sum_n a_n q^n$ be a newform of weight $1$ for $\Gamma_1(N)$
with associated Deligne--Serre representation $\rho_f$,
a $2$-dimensional complex Galois representation of
$\Gal(\overline{\Q}/\Q)$. It has a $3$-dimensional
trace-free adjoint representation $\Ad(\rho_f)$
with finite projective image that can either be dihedral
or exotic (isomorphic to $A_4$, $S_4$, or $A_5$).
We view $f$ as a modular form with coefficients in the ring of
integers $\Z[f] = \Z[\chi_{\rho_f}]$ of the number field
$\Q(f) = \Q(\chi_{\rho_f})$ generated by the trace of $\rho_f$
(or coefficients of $f$).
Let $E/\Q$ be a finite Galois extension such that
$\Ad(\rho_f)$ factors through $\Gal(E/\Q)$ and
let $M$ be the lattice in $\Ad(\rho_f)$ generated by 
elements of the form
$2\rho_f(\sigma) - \Tr(\rho_f(\sigma))\Id_{2 \times 2}$
with $\sigma \in \Gal(E/\Q)$.
The conjectures of Stark and
Harris--Venkatesh relate invariants of $f$
to elements in a (dual) space of units:
\[
	\calu\big(\Ad(\rho_f)\big) := \Hom_{\Gal(E/\Q)}
		\paren{M, \calo_E^\times \otimes \Z\sbrac{\chi_{\rho_f}}}.
\]
This is a rank-$1$ $\Z[\chi_{\rho_f}]$-module
that, like $M$, does not depend on the choice of $E$
(as detailed in \cite[Lemma 2.1]{hv}, \cite[Corollary 2.6]{horawa}).
Letting $\Z_{(p-1)}$ be subring of $\Q$ with denominators
coprime to $p-1$, there are regulator maps with values in $\R$ and
$\F_p^\times$ for any $p \nmid 6N$,
\begin{align*}
	\Reg_\R: \calu\big(\Ad(\rho_f)\big) \otimes \Q
		&\lra \R \otimes \Q\sbrac{\chi_{\rho_f}} \\
	\Reg_{\F_p^\times}: \calu\big(\Ad(\rho_f)\big) \otimes \Z_{(p-1)}
		&\lra \F_p^\times \otimes \Z\sbrac{\chi_{\rho_f}, \frac{1}{6N}},
\end{align*}
obtained from evaluation maps on distinguished elements $x_w \in M$
at an archimedean place $w$ or a non-archimedean place $w$ above $p$
(see \cite{dhrv,zhang-hv}).

Let $f^* = \sum_n \overline{a}_n q^n$ be the complex conjugate form of $f$
given by $f^*(z) := \overline{f(-\overline{z})}$,
and let $f_*$ be the Serre-dual of $f^*$
given by $\abrac{f^*, f_*}_\SD = 1$. 
The Harris--Venkatesh conjecture predicts that there
are global objects $u \in \calu(\Ad(\rho_f))$ and $m \in \N$
such that $m \cdot T_{p, N}(f) = \Reg_{\F_p^\times}(u) \cdot f_*$
simultaneously for every prime $p \nmid 6N$.
Using the Shimura class $\mfs_p \in \Hom(S_2(\Gamma_0(p)), \F_p^\times \otimes \Z[1/6])$
as described in \cite[Section 3.1]{hv},
\cite[Section 1.1]{dhrv}, and \cite{zhang-hv}, the
prescribed equation can be written more explicitly as
\[
	m \cdot \mfs_p\paren{\Tr_{\Gamma_0(p)}^{\Gamma_0(p) \cap \Gamma_1(N)}\big(f(z)f^*(pz)\big)}
		= \Reg_{\F_p^\times}\paren{u}.
\]
As noted in \cite[Section 1.2]{dhrv}, both sides are independent of the choice of
a finite place above $p$. For notational convenience,
we will use the Harris--Venkatesh norm defined for any prime $p \nmid 6N$:
\[
	\norm{f}^2_{\F_p^\times} :=
		\mfs_p\paren{\Tr_{\Gamma_0(p)}^{\Gamma_0(p) \cap \Gamma_1(N)}\paren{f(z)f^*(pz)}}
		\in \F_p^\times \otimes \Z\sbrac{\chi_{\rho_f}, \frac{1}{6N}}.
\]

The Stark conjecture \cite{stark-1971,stark-1975,stark-1976,stark-1980}
for $\Ad(\rho_f)$ predicts that
there is a \textit{unique}
(up to multiplication by roots of unity) element
$u_\stark \in \calu(\Ad(\rho_f)) \otimes \Q$ that 
describes the leading coefficient of the Artin $L$-functions:
\[
	L'\big(\Ad(\rho_f), 0\big) = \Reg_\R\paren{u_\stark},
\]
and compatibly so with Galois conjugation of $\Ad(\rho_f)$.
As observed by Stark \cite[Section 6]{stark-1975},
the Petersson norm
\[
	\norm{f}_\R^2
		= \int_{X_0(N)} \verts{f}^2 \frac{dxdy}{y} \in \R,
\]
is proportional to $L'(\Ad(\rho_f), 0)$
by a constant involving a finite product of Euler factors
only depending on $f$. Both the
Stark regulator and the Petersson norm
vary with the choice of complex embedding
(i.e. Galois conjugation of $\Ad(\rho_f)$)
in the same way.

Most newforms of weight $1$ are dihedral,
in which case there is a quadratic number field $K$
such that $\rho_f$ is induced from a character $\chi$ of
$\Gal(\overline{K}/K)$.
When $f$ is imaginary dihedral ($K/\Q$ is imaginary quadratic),
the Stark conjecture for $\Ad(\rho_f)$ was
proved in the fourth part of Stark's series of papers
\cite[Theorem 2]{stark-1980}
and the Harris--Venkatesh conjecture was proved
in the first part of this series \cite[Theorem 4]{zhang-hv}. 
We prove that these two theorems are compatible,
precisely relating the $u_\stark$ of the Stark conjecture
with the $u$ of the Harris--Venkatesh conjecture.
\begin{theorem}
	\label{thm:hvs-imaginary}
	Let $f$ be an imaginary dihedral newform of weight $1$ and level $N$.
	There is a unique (up to multiplication by roots of unity)
	$u_f \in \calu(\Ad(\rho_f)) \otimes \Q$
	such that,
	\begin{enumerate}[(a)]
		\item
			\[
				\norm{f}^2_{\R} = \Reg_\R\paren{u_f},
			\]
			compatible with conjugations of $f$ under $\Aut(\C)$;
		\item for each sufficiently large prime $p$,
			\[
				\norm{f}^2_{\F_p^\times} = \Reg_{\F_p^\times}\paren{u_f}.
			\]
	\end{enumerate}
\end{theorem}

\begin{remark}
	\label{rem:hvs-equivalences}
	Condition (a) of
	Theorem \ref{thm:hvs-imaginary} is equivalent
	to the Stark conjecture for $\Ad(\rho_f)$
	because the ratio of $u_f$ and $u_\Stark$ corresponds
	to the proportionality of $\norm{f}^2_{\R}$ and $L'(\Ad(\rho_f), 0)$.
	Without the uniqueness of $u_f$, condition (b) is equivalent
	to the Harris--Venkatesh conjecture
	because $u$ can be chosen to be any multiple of $u_f$
	and both sides vanish mod $p$ for primes smaller
	than some minimal prime $p_0$
	when both sides are multiplied by a positive integer
	$m \in \paren{\prod_{p < p_0} p} \, \N$.
\end{remark}

\begin{remark}
	\label{rem:horawa-hvs-imaginary}
	Horawa \cite[Conjecture 1.1]{horawa} reformulates
	the Harris--Venkatesh conjecture and the Stark conjecture
	in the motivic language of Venkatesh \cite{venkatesh-2}
	and Prasanna--Venkatesh \cite{prasanna-venkatesh}.
	In that context,
	Theorem \ref{thm:hvs-imaginary} gives a strengthening of
	the $F = \Q$ imaginary dihedral case of
	\cite[Conjecture 1.1 and Theorem 1.2]{horawa}
	by removing the ambiguity
	of $\calo_E^\times$ and $\Q(\chi_{\rho_f})^\times$.
	Theorem \ref{thm:hvs-imaginary} demonstrates
	the existence of an action of
	$U_f^\vee$ on $H^*(X_{\Z[1/N]}, \omega)_f$ such that,
	\begin{enumerate}[(a)]
		\item the $(U_f^\vee \otimes \C)$-action coincides with
			$u \in U_f^\vee \otimes \Q$ sending
			$f \mapsto \frac{\omega_f^\infty}{\log \verts{u}}$;
		\item the $(U_f^\vee \otimes \F_p)$-action coincides with
			the action of the derived Hecke operator.
	\end{enumerate}
\end{remark}

The key technical input needed for the proof of
Theorem \ref{thm:hvs-imaginary} is a precise
two-variable Siegel--Weil formula on $\PGL_2$.
The classical Siegel--Weil formula is an identity between
an integral of a theta function and a value of a classical Eisenstein series,
dating back to the seminal works of Siegel \cite{siegel} and Weil \cite{weil1,weil2}.
It has been extended to residues (e.g. second-term identities),
to points outside the region of absolute convergence,
and to different groups since then
(e.g. \cite{kudla-rallis-1,kudla-rallis-2,kudla-rallis-3,ikeda,ichino-sw,yamana-1,yamana-2,yamana-3,gan-qiu-takeda}).
This includes a general framework for the special orthogonal
groups corresponding to the Siegel--Weil formula
for $\PGL_2 \times \PGL_2$, but the following
explicit form has not been written down before
to the author's knowledge.
Let $\sbrac{\PGL_2} := \PGL_2(\Q) \bs \PGL_2(\A)$
and let $E_s = \sum_{n = -1}^\infty \sce_n s^n$
be the automorphic avatar of the classical Eisenstein
series with its Laurent expansion.
Our two-variable Siegel--Weil formula for the
Eisenstein series $E_s$
uses the theta lifting (and its dual) from the space $\cala_2$
of automorphic forms perpendicular to constant functions
to the space $\cala_1$ of functions that decay at cusps:
\[
	\begin{tikzcd}[row sep=tiny]
		\cala_2\paren{\sbrac{\PGL_2} \times \sbrac{\PGL_2}}
			\arrow[r, "\Theta"] & \cala_1\paren{\sbrac{\PGL_2}} \\
		\cala_1^*\paren{\sbrac{\PGL_2}}
			\arrow[r, "\Theta^*"] & \cala_2^*\paren{\sbrac{\PGL_2} \times \sbrac{\PGL_2}}.
	\end{tikzcd}
\]

\begin{theorem}
	\label{thm:theta-dual-eisenstein-2}
	For the Eisenstein series $E_s$,
	\[
		\Theta^*\paren{E_s} = E_s \otimes E_s.
	\]
	Furthermore, for any $n \geq -1$,
	\[
		\Theta^*(\sce_n) = \sum_{i + j = n} \sce_i \cdot \sce_j.
	\]
\end{theorem}

A special case of Theorem \ref{thm:theta-dual-eisenstein-2}
is the identity
\[
	\Theta^*(\one) = \one \otimes \sce_0 + \sce_0 \otimes \one,
\]
which extends the formula of
Darmon--Harris--Rotger--Venkatesh \cite[Theorem 5.4]{dhrv}
from the setting of higher Eisenstein series developed
by Merel \cite{merel2} and Lecouturier \cite{lecouturier}
to Laurent coefficients of Eisenstein series.
This identity can be applied to
explicitly calculate Rankin--Selberg convolutions of certain theta series,
including those corresponding to the two-variable optimal forms defined
in \cite[Section 6]{zhang-hv}.
By the theory of Rankin--Selberg periods
and multiplicity-one arguments developed in \cite{zhang-hv},
these calculations then give rise to units satisfying
the Harris--Venkatesh conjecture.
We then precisely pin down the unit $u_f$
using the Kronecker limit formula
for the Laurent coefficient $\sce_0$
and a refinement of Stark's formula \cite[Lemma 7]{stark-1980}
relating $L$-values with the
explicit elliptic unit $u_\xi$ defined in \cite[Equation 8.2]{zhang-hv}
for ring class characters $\xi$.
As in Stark's conjecture
\cite[Conjecture 1]{stark-1980}
and Stark's formula \cite[Lemma 7]{stark-1980}
for imaginary quadratic fields $K$,
we focus on ring class characters of conductor $c>1$
due to the non-vanishing of $\zeta_K(0)$.
\begin{proposition}
	\label{prop:stark2}
	Let $K$ be an imaginary quadratic number field.
	If $\xi$ is a ring class character of conductor $c > 1$,
	then its Hecke $L$-function satisfies
	\[
		L'(\xi, 0) = -\frac{1}{6m(\xi)} \log\paren{u_\xi}.
	\]
\end{proposition}

Finally, we propose that the derived Hecke operators
of Harris--Venkatesh give the mod-$p$ data of
the Stark unit for \textit{all} newforms of weight $1$
and almost all primes $p$.

\begin{conj}
	\label{conj:hvs}
	Let $f$ be a newform of weight $1$ and level $N$.
	There is a unique (up to multiplication by roots of unity)
	$u_f \in \calu(\Ad(\rho_f)) \otimes \Q$
	such that,
	\begin{enumerate}[(a)]
		\item
			\[
				\norm{f}^2_{\R} = \Reg_\R\paren{u_f},
			\]
			compatible with conjugations of $f$ under $\Aut(\C)$;
		\item for each sufficiently large prime $p$,
			\[
				\norm{f}^2_{\F_p^\times} = \Reg_{\F_p^\times}\paren{u_f}.
			\]
	\end{enumerate}
\end{conj}

\begin{remark}
	\label{rem:totally-real}
	As with Theorem \ref{thm:hvs-imaginary},
	Conjecture \ref{conj:hvs} is a strengthening
	of the $F = \Q$ case of
	\cite[Conjecture 1.1 and Theorem 1.2]{horawa}.
	Horawa \cite[Conjecture 1.1]{horawa} also considers
	totally real base fields $F$;
	we expect that there is a version of Conjecture \ref{conj:hvs}
	for totally real $F$.
\end{remark}

There is numerical evidence for
the Stark and Harris--Venkatesh conjectures
in the exotic cases (e.g. \cite{chinburg,fogel,jrs,marcil}),
but the Stark and Harris--Venkatesh conjectures are still open for
most exotic newforms. Conjecture \ref{conj:hvs}
suggests that the Stark and Harris--Venkatesh conjectures
should be at least as difficult as the other.
One can speculate that in this picture,
one could approach Stark units
by understanding $\norm{f}^2_{\F_p^\times}$ at all but finitely many
primes $p$.

While the construction of Stark units for exotic newforms
still eludes current techniques, the theory of optimal forms
developed in \cite{zhang-hv}
can be extended to a large class of exotic newforms.
This is the subject of the third part of this series
\cite{zhang-rs}, where the minimal prime $p_0$ of
condition (b) of Theorem \ref{thm:hvs-imaginary}
and Conjecture \ref{conj:hvs} is
precisely described using local data
and where the local theory is extended to
\textit{locally dihedral forms} (including exotic weight-$1$
newforms $f$ of odd level or $2$-ordinary
Deligne--Serre representation $\rho_f$).

\subsection{Acknowledgements}

This work would not have been possible without the support 
and encouragement of Michael Harris,
who advised the thesis that this article extends,
and Dick Gross, who asked insightful questions about the units
in my thesis and encouragingly described the resolution given herein
as showing that there is "one unit to rule them all".
I would also like to thank Henri Darmon, Emmanuel Lecouturier,
Yingkun Li, Lo\"{i}c Merel, and Akshay Venkatesh for
illuminating conversations,
as well as Aleksander Horawa
for pointing out the connections in Remark \ref{rem:totally-real}.

This material is based on work
supported by the National Science Foundation
under Grant No. DGE-1644869 and Grant No. DMS-2303280.


\numberwithin{equation}{section}
\numberwithin{theorem}{section}


\section{A formula for theta liftings of Eisenstein series}
\label{sec:eisenstein}
\subsection{\texorpdfstring{$\theta$}{theta} lifting}
\label{subsec:theta}
Let $V = (M_2(\Q), \det)$ with a
split quadratic form of dimension $4$.
Let $\GO(V)$ denote 
the group of orthogonal similitudes of $V$, i.e.
the group of elements $g \in \GL(V)$ such that
$\pair{gx, gy} = \nu(g)\pair{x, y}$ for some $\nu(g) \in \Q^\times$.
For $g \in \GO(V)$, we denote the conjugate by
$\overline{g} = \nu(g) g^{-1}$.
Let $\GSO(V)$ be the connected component
of the identity of $\GO(V)$,
\[
	\GSO(V) = (\GL_2 \times \GL_2) / \Delta \G_m.
\]
With square brackets, $[G]$ denotes the quotient
$G(\Q) \bs G(\A)$.

We have the Weil representation $r$ of
$\GL_2(\A) \times_{\G_m} \GSO(V_\A)$ on $\cals(V_\A)$ and
a theta kernel for each
$\Phi \in \cals(V_\A)$ (cf. \cite[Section 2.1]{yzz}
and \cite[Section 5.1]{zhang-hv}),
\[
	\theta(g, h, \Phi) := \sum_{x \in V} r(g, h) \Phi(x).
\]
We want to consider the theta liftings between forms on
$[\GL_2]$ and $[\GSO(V)]$.
For any automorphic function $\varphi$ on
$[\GSO(V)]$ and $\Phi \in \cals(V_\A)$,
we have a function,
\[
	\theta(g, \varphi, \Phi)
		= \int_{[\SO(V)]} \theta\paren{g, hh_0, \Phi}
			\varphi\paren{hh_0} dh,
\]
where $h_0 \in \GSO(V)$ is an element with norm $\det g$ to
ensure that $h h_0$ is in $\GSO(V)$ and has the same norm as
$g$.

Let $y \in \A^\times$ and $h_1 \in \GSO (V)$ such that $\nu(h_1) = y$.
Then we have for
$a(y) := \begin{psmallmatrix}y & 0 \\ 0 & 1\end{psmallmatrix}$,
\begin{align}
	\label{eq:theta-ay}
	\theta(a(y)g, \varphi, \Phi)
		&= \int_{[\SO(V)]} \theta\paren{a(y)g, hh_1h_0, \Phi}
			\varphi\paren{hh_1h_0} dh \\
		&= \int_{[\GSO(V)]} \sum_{x \in V} r\paren{g, h_0}
			\Phi\paren{\overline{h}_1 \overline{h} x} \verts{y}
			\varphi\paren{hh_1h_0} dh. \nonumber
\end{align}
Here we used the fact that for
$d(y) := \begin{psmallmatrix} 1 & \\ & y \end{psmallmatrix}$
and
$m(y) := \begin{psmallmatrix} y & \\ & y^{-1} \end{psmallmatrix}$
(cf. \cite[Section 5.1]{zhang-hv}),
\begin{align*}
	r\paren{a(y), h_1} \Phi(x)
	&= r\paren{m(y)} r\paren{d(y), h_1} \Phi(x) \\
	&= \verts{y}^2 r\paren{d(y), h_1} \Phi(yx) \\
	&= \verts{y}^2 \verts{y}^{-1} \Phi\paren{h_1^{-1}y x} \\
	&= \verts{y} \Phi\paren{\overline{h}_1 x}.
\end{align*}
Every term in the sum has exponential decay in $y$, except the term
for $x=0$:
\[
	r\paren{g, h_0} \Phi(0) \int_{[\SO (V)]} \varphi\paren{hh_1h_0}dh.
\]
Let $\omega$ be a character of $\Q^\times \bs \A^\times$.
If $\varphi$ has central character $\omega^{-1}$,
then $\theta(g, \varphi, \Phi)$ has central character $\omega^{-1}$.

Let $\calm([\GL_2], \omega^{\pm 1})$
be the space of automorphic functions on $[\GL_2]$
with central character $\omega^{\pm 1}$ and rapid decay at cusps, and 
let $\calm([\GSO(V)], \omega^{\pm 1})$
be the space of automorphic functions $\varphi$ on $[\GSO(V)]$
with central character $\omega^{\pm 1}$ and rapid decay at cusps
such that $\int_{[\GSO(V)]} \varphi(hh_0) dh=0$
for all $h_0 \in [\GSO(V)]$.
Then we have defined a theta lifting:
\[
	\theta: \calm\paren{[\GSO(V)], \omega^{-1}} \otimes \cals\paren{V_\A}
		\lra \calm\paren{\sbrac{\GL_2}, \omega^{-1}}.
\]
We also have the dual theta lifting,
\[
	\theta^*: \calm^*\Big(\sbrac{\GL_2}, \omega\Big) \otimes \cals\paren{V_\A}
		\lra \calm^*\Big([\GSO(V)], \omega\Big).
\]

The main goal of this section is to compute
$\theta^*(h, E, \Phi)$ for the Siegel Eisenstein series,
\[
	E(g, f, s) = \sum_{\gamma \in P(\Q) \bs \GL_2(\Q)} f(\gamma g),
\]
where $f$ is a function on $\GL_2(\A) \times \C$ such that,
\[
	f\paren{\begin{pmatrix}a & b \\ 0 & d \end{pmatrix} k}
		= \mu_1(a) \mu_2(d) \verts{\frac{a}{d}}^{s} f(k),
\]
with $\mu_1, \mu_2$ two characters of $\Q^\times \bs \A^\times$
such that $\mu_1 \cdot \mu_2 = \omega$.
When $\Re s \gg 0$, the Eisenstein series
$E(g, f, s)$ is absolutely convergent.
Let $H_s(\mu_1, \mu_2)$ be the space of such functions $f$ on
$\GL_2(\A) \times \C$,
and let $\pi_s(\mu_1, \mu_2)$ be the space of Eisenstein series
forms $E(g, f, s)$ with $f \in H_s(\mu_1, \mu_2)$. 
\begin{theorem}
	\label{thm:theta-dual-eisenstein}
	$\theta^*$ defines a map,
	\[
		\theta^*: \pi_s\big(\mu_1, \mu_2\big) \otimes \cals\paren{V_\A}
			\lra \pi_s\big(\mu_1^{-1}, \mu_2^{-1}\big)
				\otimes \pi_s\big(\mu_2, \mu_1\big).
	\]
\end{theorem}

\begin{proof}
	Fix a $\Phi \in \cals (V_\A)$ and $f \in H_s(\mu_1, \mu_2)$.
	We can then write $E_s(g) := E(g, f, s)$ and
	$\theta^*(E) := \theta^*(E, \Phi)$.
	For any $\varphi \in \calm([\GSO(V)], \omega)$,
	unfolding the integral gives,
	\begin{align*}
		\pair{\theta^*\paren{E_s}, \varphi}
			&= \int_{\sbrac{\PGL_2}} \theta(g, \varphi, \Phi) E_s(g) dg \\
			&= \int_{P(\Q)Z(\A) \bs \GL_2(\A)} \theta(g, \varphi, \Phi) f(g) dg.
	\end{align*}
	Now we use the measure $dg = \verts{y}^{-1} dx d^\times ydk$
	for the Iwasawa decomposition $g=z n(x) a(y) k$ with
	$n(x) := \begin{psmallmatrix} 1 & x \\ & 1 \end{psmallmatrix}$
	and $k \in \K$,
	the standard maximal compact subgroup of $\SL_2(\Z)$,
	to obtain,
	\begin{align*}
		\pair{\theta^*\paren{E_s}, \varphi}
			&= \int_{P(\Q)Z(\A) \bs \GL_2(\A)} \theta\big(n(x)a(y)k, \varphi, \Phi\big)
				\mu_1(y) \verts{y}^{s-1} f(k) dx d^\times y dk.
	\end{align*}

	The integration over $dx$ amounts to replacing
	$\theta(g, \varphi, \Phi)$ by, 
	\[
		I_1 = \int_{[\SO(V)]}\sum_{x \in V_{\det =0}} r\big(g, h_0\big)
			\Phi\big(h^{-1}x\big) \varphi\big(hh_0\big) dh.
	\]
	Since $\varphi \in \calm([\GO (V)], \omega^{-1})$,
	there is no contribution from the term at $x = 0$.
	Therefore we can replace $V_{\det = 0}$ by
	the subset $X \subset V$ of non-zero elements with $\det = 0$.
	For integration over $\mu_1^{-1}(y)\verts{y}^{-1}d^\times y$,
	we use Equation \ref{eq:theta-ay} to replace $I_1$ by
	\[
		\int_{[\SO(V)]} \sum_{x \in X}
			r\paren{g, h_0} \Phi\big(\overline{h}_1 \overline{h} x\big)
			\varphi\paren{hh_1h_0} \verts{y}dh,
	\]
	where $h_1 \in \GSO(V)$ so that $\nu(h_1)=y$.
	Finally for the integration over $k$, we take a section 
	$h_k := (k, 1)$; therefore
	\begin{align*}
		\pair{\theta^*\paren{E_s}, \varphi}
		= \int_{\K}\int_{\Q^\times \bs \A^\times}
			\int_{[\SO(V)]} \sum_{x \in X} r\paren{k, h_k}
			\Phi\big(\overline{h}_1 \overline{h}x\big)
			\varphi\big(hh_1h_k\big) dh \mu_1^{-1}(y) \verts{y}^{s}dy f(k)dk
	\end{align*}
	The inside two integrals can be calculated together
	as a single integral,
	\begin{align*}
		\int_{[\SO(V)]} \sum_{x \in X} r\paren{k, h_k}
				&\Phi\big(\overline{h}_1 \overline{h} x\big)
				\varphi\big(hh_1h_k\big) dh \mu_1^{-1}(y) \verts{y}^{s}dy \\
			&= \int_{[\GO(V)]} \sum_{x \in X} r\big(k, h_k\big)
				\Phi\big(\overline{h}x\big) \varphi(h)
				\mu_1^{-1}\big(\nu (h)\big) \verts{\nu(h)}^{s} dh \\
			&= \int_{[\GSO(V)]} \sum_{x \in X} r\big(k, h_k\big)
				\Phi\big(\overline{h}x\big) \verts{\nu(h)}^{s}
				\mu_1^{-1}\big(\nu(h)\big) \varphi(h) dh\\
	\end{align*}
	Define $\wt{\Phi} \in \cals(V_\A)$ by,
	\[
		\wt{\Phi} := \int_\K  f(k) r(k, h_k) \Phi dk.
	\]
	Then we have,
	\[
		\pair{\theta^*(E_s), \varphi}
			= \int_{[\GSO]} \sum_{x \in X} \wt{\Phi} \big(\overline{h}x\big)
				\verts{\nu(h)}^{s} \mu_1^{-1}\big(\nu(h)\big) \varphi(h) dh.
	\]
	Integrate over the center to obtain,
	\begin{align*}
		\theta^*(E_s)
			= &\mu_1^{-1}\big(\nu(h)\big) \verts{\nu(h)}^{s}
				\int_{\Q^\times \bs \A^\times} \sum_{x \in X}
				\wt{\Phi}\big(z \overline{h}x\big)
				\mu_2\mu_1^{-1}(z) \verts{z}^{2s} dz \\
			= &\sum_{x \in X/\Q^\times}
				\mu_1^{-1}\big(\nu(h)\big) \verts{\nu(h)}^{s}
				\int_{\A^\times} \wt{\Phi}\big(z \overline{h}x\big)
				\mu_2\mu_1^{-1}(z) \verts{z}^{2s} dz.
	\end{align*}
	A further change variable $z \mapsto z\nu(h)^{-1}$ gives,
	\[
		\theta^*(E_s)
			= \sum_{x \in X/\Q^\times}
			\mu_2^{-1}\big(\nu(h)\big) \verts{\nu(h)}^{-s}
			\int_{\A^\times} \wt{\Phi}\big(zh^{-1}x\big)
			\mu_2\mu_1^{-1}(z) \verts{z}^{2s} dz.
	\]

	Now we write this as an Eisenstein series.
	For this, we know that $X/\Q^\times$ is the orbit of $\Q x_0$ with 
	$x_0 = \begin{psmallmatrix} 0 & 1 \\ 0 & 0 \end{psmallmatrix}$ under
	$\GSO(V)= \GL_2 \times \GL_2 / \Delta(\G_m)$ with stabilizer
	$P \times P$, where $P$ is the subgroup of upper triangular matrices
	in $\GL_2$.
	Thus we can write 
	\[
		\theta^*(E_s)(h_1, h_2)
			= \sum_{(\gamma_1, \gamma_2) \in
				(P \bs \GL_2) \times (P \bs \GL_2)}
				F_s(\gamma_1 h_1, \gamma_2 h_2),
	\]
	where $F_s(h_1, h_2)$ is the function on
	$\GL_2(\A) \times \GL_2(\A)$,
	\[
		F_s(h_1, h_2)
			:= \mu_2^{-1}\paren{\det h_1h_2^{-1}}
				\verts{\det(h_1h_2^{-1})}^{-s}
				\int_{\A^\times} \wt{\Phi}\big(zh_1^{-1}x_0 h_2\big)
				\mu_2\mu_1^{-1}(z) \verts{z}^{2s} dz.
	\]

	We want to study the behavior of $F_s$ under left-translation by
	elements in $P(\A) \times P(\A)$.
	For the diagonal elements
	$m_1 = \begin{psmallmatrix} a_1 & 0 \\ 0 & b_1 \end{psmallmatrix}$
	and $m_2 = \begin{psmallmatrix} a_2 & 0 \\ 0 & b_2 \end{psmallmatrix}$,
	\begin{align*}
		F_s\big(m_1h_1, \, m_2h_2\big)
			= \, &\mu_2^{-1}\big(a_1 b_1 a_2^{-1} b_2^{-1}\big)
				\, \verts{a_1b_1a_2^{-1}b_2^{-1}}^{-s} \\
				&\cdot \mu_2\mu_1^{-1}\big(a_1 b_2^{-1}\big)
				\, \verts{a_1 b_2^{-1}}^{2s} \, F_s\big(h_1, h_2\big) \\
			= \, &\mu_1^{-1}\big(a_1\big)\mu_2^{-1}\big(b_1\big)
				\, \verts{\frac{a_1}{b_1}}^{s} \\
				&\cdot \mu_2\big(a_2\big) \mu_1\big(b_2\big)
				\, \verts{\frac{a_2}{b_2}}^{s} \, F_s\big(h_1, h_2\big).
	\end{align*}
	This implies that
	$F_s \in H_s(\mu_1^{-1}, \mu_2^{-1}) \otimes H_s(\mu_2, \mu_1)$.
	Therefore, $\theta^*(E) \in \pi_s(\mu_1^{-1}, \mu_2^{-1})
	\otimes \pi_s(\mu_2, \mu_1)$.
\end{proof}

\subsection{\texorpdfstring{$\Theta$}{Theta} lifting}
Take $\Phi = \Phi_\infty \otimes \Phi^\infty$
to be the standard Schwartz function
on $V(\A) = M_2(\A)$
(similar to \cite[Section 5.2]{zhang-hv} and
\cite[Definition 5.5]{zhang-hv} which 
define standard Schwartz functions relative to a character
and to an Eichler order $\calo$ respectively
rather than maximal orders):
\begin{itemize}
	\item $\Phi_\infty(x) = e^{-2\pi \verts{x}^2}$;
	\item $\Phi^\infty(x)$ is the characteristic
		function of $M_2(\hat{\Z})$;
\end{itemize}
and let $\omega = 1$. We have a big theta lifting
from $\cala_2$, the space of automorphic forms perpendicular to
constant functions, to $\cala_1$, the space of functions that
rapidly decay at cusps:
\[
	\Theta: \cala_2\paren{\sbrac{\PGL_2} \times \sbrac{\PGL_2}}
		\lra \cala_1\paren{\sbrac{\PGL_2}}.
\]
Denote the dual lifting,
\[
	\Theta^*: \cala_1^*\paren{\sbrac{\PGL_2}}
		\lra \cala_2^*\paren{\sbrac{\PGL_2} \times \sbrac{\PGL_2}}.
\]

We will now apply Theorem \ref{thm:theta-dual-eisenstein}
to the standard Siegel Eisenstein series in
the case that $\mu_1 = \mu_2 = 1$ and
$f$ is the function with value $1$ on $\K$:
\[
	E_s(g) = \zeta^*(2s) \sum_{\gamma \in P(\Q) \bs \GL_2(\Q)}
		f(\gamma g),
\]
which is the automorphic avatar of the classical Eisenstein series,
\[
	E(z, s) = \zeta^*(2s) \sum_{\gamma \in \Gamma_\infty \bs \SL_2(\Z)}
			\big(\Im (\gamma z)\big)^s.
\]
Consider the Laurent series of $E_s$ at $s = 0$,
with $E_s = \sum_{n=-1}^\infty \sce_n s^n$ and
$\sce_{-1} = -\frac{1}{2}$.
We establish the formulas of Theorem \ref{thm:theta-dual-eisenstein-2}
for the dual theta lifting evaluated on the Eisenstein series and
the coefficients $\sce_n$, giving an analogue of the
higher Eisenstein series formula of \cite[Theorem 5.4]{dhrv}
and the Siegel--Weil formula.

\begin{proof}[{Proof of Theorem \ref{thm:theta-dual-eisenstein-2}}]
	By the proof of Theorem \ref{thm:theta-dual-eisenstein},
	we know that $\Theta^*(E_s)$ is the
	Eisenstein series for $\PGL_2 \times \PGL_2$ defined by 
	\[
		F_s\big(h_1, h_2\big)
			:= \zeta^*(2s) \, \verts{\det h_1h_2^{-1}}^{-s}
				\int_{\A^\times} \Phi\big(zh_1^{-1}x_0 h_2\big)
				\verts{z}^{2s} dz.
	\]
	It is clear that $F_s$ is spherical, so it is equal to 
	$cf(h_1)f(h_2)$ for some constant $c$. To compute $c$,
	we take $h_1 = h_2 = 1$ to obtain 
	$c = \zeta^*(2s)^2$.
	Thus we have the first formula,
	\[
		\Theta^*(E_s) = E_s \otimes E_s.
	\]

	The second formula then follows from the first one.
	For any $\varphi \in \cala_2([\PGL_2]\times [\PGL_2])$, we have
	\begin{align*}
		\sum_{n \geq -1}\pair{\Theta^*(\sce_n), \varphi}s^n
			&= \pair{\Theta^*(E_s), \varphi} \\
			&= \pair{E_s \otimes E_s, \varphi}\\
			&= \pair{\sum_{n \geq -2}
				\paren{\sum_{i + j = n} \sce_i \sce_j s^n}, \varphi} \\
			&= \sum_{n \geq -1}
				\pair{\sum_{i + j = n} \sce_i\sce_j, \varphi} s^n.
	\end{align*}
	In the last step,
	we have used the fact that $\varphi$ is perpendicular to
	constant functions. 
\end{proof}


\section{Stark units for imaginary quadratic fields}
\label{sec:units}

\subsection{The unit of Stark}
\label{sec:Stark-unit}
In the imaginary dihedral (i.e. imaginary quadratic) setting,
the Stark conjecture is known
due to the original work of Stark \cite{stark-1980}.
There is a decomposition
$\Ad(\rho_f) \cong \eta \oplus \ind_{G_K}^{G_\Q}(\xi)$
of the (trace-free) adjoint representation
into the direct sum of
a quadratic character $\eta$
and the induced representation of a
character $\xi$ of $G_K = \gal(\overline{K}/K)$
for an imaginary quadratic number field $K$.
Complex conjugation acts on
$\Ad(\rho_f)$
with eigenvalues $1, -1, -1$
so the space of units
$\calu(\Ad(\rho_f))$ has rank $1$
by the Dirichlet unit theorem
(cf. \cite[Lemma 2.1]{hv}, \cite[Section 1]{dhrv});
here the rank,
\[
	r\big(\Ad(\rho_f)\big) :=
		\sum_{v \mid \infty} \dim\paren{\Ad(\rho_f)^{\frob_v}},
\]	
is also the order of vanishing of the Artin $L$-function
$L(\Ad(\rho_f), s)$ at $s = 0$
(cf. \cite[Section 1]{stark-1975}, \cite[Section 1]{tate-stark-1982},
Proof of Proposition \ref{prop:stark-dihedral}).

First, we recall an original result of Stark \cite{stark-1980}
that proves a weaker but explicit version of his conjecture
for CM characters.
Let $\mfc$ be the conductor of $\xi$, i.e.
the maximal ideal of $\calo_K$ such that 
$\xi$ is trivial on
$(1 + \wh{\mfc})^\times \subset \wh{\calo_K}^\times$
as a character on $\wh{K}^\times$.
Let $E = K(\mfc)$ be the corresponding ray class field with Galois
group $G(\mfc)$. Thus by class field theory, the Artin reciprocity map induces
an isomorphism
\[
	K^\times \bs \wh{K}^\times /(1 + \wh{\mfc})^\times \iso G(\mfc).
\]
Let $c$ be the minimal positive integer divisible by $\mfc$ and
let $w(\mfc)$ be the number of roots of unity in $k$ 
which are congruent to $1 \pmod \mfc$.
Recall that for rational numbers and
imaginary quadratic fields $K$,
Stark's conjecture
\cite[Conjecture 1]{stark-1980}
is only stated for $c>1$
due to the non-vanishing of $\zeta_K(0)$.
\begin{theorem}[{Stark \cite[Lemma 7]{stark-1980}}] 
	\label{thm:stark1}
	Let $K$ be an imaginary quadratic number field.
	If $\xi$ is a non-trivial character of
	$\gal(\overline{K}/K)$ of non-trivial conductor, then
	there is an element $\epsilon(\mfc) \in K(\mfc)^\times$
	such that,
	\[
		L'(\xi, 0) = -\frac{1}{6c w(\mfc)} \sum_{\sigma \in G(\mfc)}
			\xi(\sigma) \log\verts{\epsilon(\mfc)^\sigma}.
	\]
\end{theorem}
\begin{remark}
	Our notation here deviates from \cite{stark-1980}:
	our $\mfc$, $c$, $\sigma$ are respectively $\mff$, $f$, and $\mfc$
	in Stark's paper.
\end{remark}

Stark proves Theorem~\ref{thm:stark1}
using the second Kronecker limit formula.
For $u, v \in \R$, $z \in \calh$,
$q := e^{2\pi i z}$, and $\zeta := e^{2\pi i (uz + v)}$,
we have the Siegel function,
\[
	g(u, v, z) := -i q^{\frac{1}{12}}
		\paren{\zeta^{\frac{1}{2}} - \zeta^{-\frac{1}{2}}}
		\prod_{m = 1}^\infty
		\paren{1 - q^m \zeta} \paren{1 - q^m\zeta ^{-1}}.
\]
Again with $H_c$ the corresponding ring class field to
the conductor of $\xi$, fix
a decomposition of $\calo_{K(\mfc)} = \Z + z\Z$
and $\mfc = c(\Z + z \Z)$
with $z \in \calh \cap K(\mfc)$.
Then by \cite[Equations (9), (10), and (45)]{stark-1980}, 
\[
	\epsilon(\mfc) = g\paren{0, \frac{1}{c}, z}^{12 c}.
\]
Stark described the Galois conjugates of $\epsilon(\mfc)$ as follows.
For any $\sigma \in G$ represented by ideal $\mfa$ prime to $\mfc$,
let $\mfb$ be an ideal such that $\mfa \mfb = (\alpha(\sigma))$
for some $\alpha(\sigma) \in K^\times$
and write $\mfc \mfb = c(\sigma)(\Z + z(\sigma)\Z)$.
Then $\alpha(\sigma)/c(\sigma) = u(\sigma) z(\sigma) + v(\sigma)$ for
some rational numbers $u(\sigma)$ and $v(\sigma)$, and,
\[
	\epsilon(\mfc)^\sigma = g\paren{u(\sigma), v(\sigma), z(\sigma)}^{12c}.
\]

We can give a precise version of Stark's theorem
that is not stated by Stark \cite{stark-1980} nor Tate
\cite{tate-stark-1982, tate-stark-1984}
but which we deduce directly from
Theorem \ref{thm:stark1}.
Define $\epsilon(c) := \rmn_{K(\mfc)/H_c}(\epsilon(\mfc))$
and denote the modular discriminant by
$\Delta = q \prod_{n \geq 1} (1 - q^n)^{24}$.

\begin{proposition}
	\label{prop:stark}
	Assume that $\xi$ is a ring class character of conductor $c>1$.
	Then,
	\[
		L'(\xi, 0) = -\frac{1}{6c} \sum_{\sigma\in G(\mfc)}
			\xi(\sigma) \log\verts{\epsilon(\mfc)^\sigma},
	\]
	and,
	\[
		\epsilon(c) = \prod_{d \mid c}
			\Delta\paren{q^{d}}^{c \mu\paren{\frac{c}{d}}} \in H_c^\times.
	\]
\end{proposition}

\begin{proof}
	By Theorem \ref{thm:stark1}, we only need to compute the norm
	$\epsilon(c)$ of $\epsilon(\mfc)$.
	Let $\calo_c = \Z + c\calo_{K(\mfc)}$.
	Then $\mfc = c\calo_{K(\mfc)}$ and we have,
	\[
		\Gal\paren{K(\mfc) / H_c}
			\iso K^\times \paren{\wh{\Z} + c \wh{\calo}_K}^\times
				/ K^\times \paren{1 + \wh{\mfc}}^\times
			\iso \paren{\wh{\Z} + c \wh{\calo}_K}^\times
				/ \paren{1 + \wh{\mfc}}^\times
			\iso (\Z/c\Z)^\times.
	\]

	Thus every element in $\sigma \in \Gal(K(\mfc) / H_c)$ is
	represented by an	ideal $\mfa(\sigma) = n\calo_K$
	with $n$ coprime to $c$.
	We take $\mfb(\sigma) = m\calo_K$ with $\alpha(\sigma) = mn$.
	Write $\calo_K = \Z + z \Z$. Then $c(\sigma) = cm$, $u(\sigma) = 0$,
	and $v(\sigma) = n/c$.
	It follows that,
	\[
		\rmn_{K(\mfc) / H_c}\big(\epsilon(\mfc)\big) = 
			\prod_{n \in (\Z/c\Z)^\times}
			g\paren{0, \frac{n}{c}, z}^{12 c}.
	\]

	Let $\zeta_c = e^{2\pi i/c}$.
	Then $\rmn_{K(\mfc) / H_c}(\epsilon(\mfc))$
	is a product,
	\[
		q^{c\phi(c)}\cdot \rmn\paren{1 - \zeta_c}^{12 c}
			\prod_{m = 1}^\infty \prod_{n \in (\Z/c\Z)^\times}
			\paren{1 - q^m\zeta_c^n}^{24c},
	\]
	where $\phi$ is the Euler totient function.

	The term $m(c) := \rmn(1 - \zeta_c)$ is equal to $1$ unless $c$ is a prime,
	in which case it is equal to $c$.
	The other terms can be computed by M\"{o}bius inversion,
	$\phi(c) = \sum_{d \mid c} d \cdot \mu(c/d)$,
	\[
		\prod_{n \in (\Z/c\Z)^\times} \big(1 - \zeta_c^n T\big)
			= \prod_{d \mid c} \paren{1 - T^d}^{\mu\paren{\frac{c}{d}}}.
	\]
	Thus we obtain
	\begin{align*}
		\epsilon(c)
			:= \, &\rmn_{K(\mfc)/H_c} \big(\epsilon(\mfc)\big) \\
			= \, &m(c)^{12 c} q^{c\phi(c)}
				\prod_{m=1}^\infty \prod_{n \in (\Z/c\Z)^\times}
				\big(1 - \zeta_c^nq^m\big)^{24c} \\
			= \, &m(c)^{12c} q^{c\phi(c)} \prod_{d\mid c} \prod_{m = 1}^\infty
				\paren{1-q^{dm}}^{24 c\mu\paren{\frac{c}{d}}} \\
			= \, &m(c)^{12c} \prod_{d\mid c}
				\Delta\big(q^{d}\big)^{c\mu\paren{\frac{c}{d}}}.
	\end{align*}
\end{proof}

\subsection{The unit of Harris--Venkatesh}
\label{subsec:unit-hv}
For a ring class character $\xi$,
define
\[
	m(\xi) := 
		\begin{cases}
			v & \text{if $\Im(\xi)$ is a $v$-group,} \\
			1 & \text{otherwise.}
		\end{cases}
\]
Recall that for an auxiliary prime
$\lambda = \mfl \cdot \overline{\mfl}$ split in $K$
and coprime to the conductor of $\xi$,
\cite[Equation 8.2]{zhang-hv} defines an elliptic unit,
\begin{equation}
	\label{eq:uxi}
	u_\xi := \frac{m(\xi)}{1 - \xi\paren{\overline{\mfl}}}
		\, u_{\xi, \lambda},
\end{equation}
in terms of the elliptic unit
$u_{\xi, \lambda} \in H_c^\times \otimes \Z[\xi]$
given by Darmon--Harris--Rotger--Venkatesh
\cite[Definition 5.1]{dhrv}.
It is independent of the auxiliary
prime $\lambda$ by \cite[Proposition 8.4]{zhang-hv}.
We prove Proposition \ref{prop:stark2}, which relates
$L'(\xi, 0)$ to $-\frac{1}{6m(\xi)} \log(u_\xi)$, by relating
$u_\xi$ to the Siegel units from Section \ref{sec:Stark-unit}
and applying Proposition \ref{prop:stark}. 

\begin{proof}[{Proof of Proposition \ref{prop:stark2}}]
	Let $\tau = \Frob_\mfl \in \Gal(H_c/K)$.
	For any $d \mid c$, let $E_d$ denote the elliptic curve $\C/\calo_d$
	which is defined over $H_c$. Then we have the isogenies
	$E_d \lra E_d^\sigma$ with kernel $E[\mfl]$.
	With $\epsilon(c)$
	as in Proposition \ref{prop:stark}
	and $\eta$ the Dedekind eta function
	related to $\Delta$ by a power of $24$,
	\[
		\epsilon(c) - \epsilon (c)^\tau
			= c \sum_{d \mid c} \mu\paren{\frac{c}{d}} \eta(q^d).
	\]
	Take $\xi$-sums to obtain,
	\[
		\sum_{\sigma \in \Gal(H_c/K)}
			\big(\xi(\sigma) - \xi(\sigma\tau)\big) \epsilon(c)
			= c \sum_{d \mid c} \xi(\sigma) \eta(q^d)^\sigma.
	\]
	It follows that,
	\[
		\sum_{\sigma \in \Gal(H_c/K)} \xi(\sigma) \epsilon(c)
			= \frac{c}{1 - \xi(\tau)} \sum_{d \mid c} \mu\paren{\frac{c}{d}}
				\sum_{\sigma\in \Gal(H_c/K)}\xi (\sigma) \eta(q^d)^\sigma.
	\]
	Since $\eta(q^d)$ is in the ring class field $H_d$ corresponding
	to $\calo_d$, the last sum has a factor
	$\sum_{\sigma \in \Gal(H_c/H_d)} \xi(\sigma)$ which is zero if
	$d \neq c$ since $c$ is the conductor of $\xi$.
	Thus we have shown that,
	\[
		\sum_{\sigma \in \Gal(H_c/K)} \xi(\sigma) \epsilon(c)
			= \frac{c}{m(\xi)} \, u_\xi.
	\]
	The desired identity follows from Proposition \ref{prop:stark}.
\end{proof}

This allows us to give the explicit
relation between the unit $u_\xi \in H_c^\times \otimes \Z[\xi]$
and the Stark (dual) unit $u_\stark \in \calu(\Ad(\rho_f))$.

\begin{proposition}
	\label{prop:stark-dihedral}
	Let $K$ be an imaginary quadratic number field,
	$h_K$ be the class number of $K$,
	and $w_K$ be the number of roots of unity
	in $K$. There is a complex archimedean place
	$w$ of $H_c$ with distinguished element $x_w$
	of $\Ad(\rho_f)$ such that
	\[
		u_\stark(x_w) = \frac{h_K}{6m(\xi)w_K} \otimes u_\xi.
	\]
\end{proposition}

\begin{proof}
	From the decomposition
	$\Ad(\rho_f) = \eta \oplus \ind_{G_K}^{G_\Q}(\xi)$
	where $\eta$ is a quadratic character,
	\[
		L\big(\Ad(\rho_f), s\big) = L(\eta, s) \cdot L(\xi, s).
	\]
	Applying $L(\xi, 0) = 0$,
	Proposition \ref{prop:stark2}, and the class number formula,
	we see that $\Ad(\rho_f)$ has order of vanishing $1$
	at $s=0$ and
	\[
		L'\big(\Ad(\rho_f), 0\big)
			= L(\eta, 0) \cdot L'(\xi, 0)
			= \frac{h_K}{6m(\xi)w_K} \cdot \log(u_\xi),
	\]
	where the $\log$ is with respect to a fixed embedding
	$H_c \hookrightarrow \C$.
	Then there is a complex archimedean place $w$
	for which $u_{\stark}$ is the unique element such that
	\[
		u_{\stark}(x_w) = \frac{h_K}{6m(\xi)w_K} \otimes u_\xi.
	\]
\end{proof}

\begin{remark}
	A direct corollary of Proposition~\ref{prop:stark-dihedral} is
	\cite[Lemma~5.6]{dhrv}.
	It also implies that
	$6m(\xi)w_K \,u_\stark \in \calu(\Ad(\rho_f))$,
	since it evaluates to $h_K u_\xi$.
\end{remark}

Furthermore for finite places $v$ of $H_c$ over $\lambda$
that are inert in $K$,
\[
	u_{\stark}(x_v) = \frac{h_K}{6m(\xi)w_K} \otimes u_\xi.
\]
This is because the conjugacy class of $\frob_w$ does not change if
the archimedean place $w$ is replaced by an inert place $v$;
if $\lambda$ is inert in $K$ with a unique prime $\mfl$,
then $\mfl$ is completely split in $H_c$
and $x_{v} = x_w$.
Darmon--Harris--Rotger--Venkatesh \cite[Section 1.3]{dhrv}
observe that both sides of the Harris--Venkatesh conjecture
in the imaginary dihedral case vanish
when $\lambda$ is split in $K$,
so the results of \cite{zhang-hv}
can be rephrased in terms of $u_\stark$ instead of $u_\xi$.
With $h_K = [H_1 : K]$ for the Hilbert class field $H_1$,
we have the following reformulation of \cite[Theorem 8]{zhang-hv}
for the Harris--Venkatesh period $\calp_\hv$ (defined
in \cite[Section 2]{zhang-hv}) applied to the two-variable
optimal form $f^\opt$ (defined in \cite[Section 6]{zhang-hv}).
\begin{corollary}
	\label{cor:stark-opt}
	Let $f$ be an imaginary dihedral modular form of weight $1$ and
	level $N$, and let $f^\opt$ be the optimal form associated to $f$.
	For all primes $p \geq 5$ coprime to $N$,
	\begin{align*}
		\calp_\hv \paren{f^\opt}
			&= -\frac{\sbrac{H_c : H_1} w_K}{2 \sbrac{\PSL_2(\Z) : \Gamma_0(N)} \cdot} \Reg_{\F_p^\times} \paren{u_{\stark}}.
	\end{align*}
\end{corollary}

\begin{proof}
	By \cite[Theorem 8]{zhang-hv} (without
	taking $\log_\ell$ for $\ell \geq 5$ since we have
	inverted $6$) and the
	equality $[\Gamma(1) : \Gamma_0(N)] \cdot \calp_\hv \paren{f^\opt} = \mfs_p \paren{f^\opt(z, pz)}$,
	\begin{align*}
		\calp_\hv \paren{f^\opt}
			&= \frac{1}{\sbrac{\PSL_2(\Z) : \Gamma_0(N)}} \mfs_p \paren{f^\opt(z, pz)} \\
			&= -\frac{\sbrac{H_c : K}}{12m(\xi) \sbrac{\PSL_2(\Z) : \Gamma_0(N)}} \Reg_{\F_p^\times}\paren{u_\xi}.
	\end{align*}
	Then apply Proposition \ref{prop:stark-dihedral}.
\end{proof}


\section{Global Rankin--Selberg periods}
\label{sec:rs-opt}

Let $f_1$ and $f_2$ be two homomorphic cusp forms of weight $1$
for a congruence subgroup $\Gamma$ of $\PSL_2(\Z)$
with central character $\omega$ and $\omega^{-1}$ respectively.
Then we define their Rankin--Selberg period
(or Rankin--Selberg pairing),
\[
	\calp_\rs\paren{f_1 \otimes f_2}
		:= \frac{1}{\sbrac{\PSL_2(\Z) : \Gamma}} \int_{X(\Gamma)}
			f_1(z) f_2(-\overline{z}) y d\mu
\]
where $d\mu = \frac {dxdy}{y^2}$.
This definition does not depend on the choice of $\Gamma$
and extends to a pairing 
on the space $S_1(\omega)$ and $S_1(\omega^{-1})$
of forms of weight $1$ with central character
$\omega^{\pm 1}$:
\[
	\calp_\rs: S_1\big(\omega\big) \otimes S_1\big(\omega^{-1}\big)
		\lra \C.
\]

In \cite[Equations 1.3 and 1.4]{zhang-hv}, we
gave the correspondence between automorphic forms of weight $1$
and modular forms of weight $1$.
The transform $z \mapsto -\overline{z}$ corresponds to
right-multiplication by $\epsilon_\infty$ 
($\epsilon := \begin{psmallmatrix} -1 & 0 \\ 0 & 1 \end{psmallmatrix}$
at each archimedean place, this is called $\eta$ in \cite[p. 11]{jacquet}).
Thus for two automorphic forms $\varphi_1, \varphi_2$ of weight $1$
with opposite central character,
we can define their global Rankin--Selberg period,
\[
	\calp_\rs \paren{\varphi_1 \otimes \varphi_2}
		:= \frac{1}{\sbrac{\PSL_2(\Z) : \Gamma_0(N)}} \int_{\sbrac{\PGL_2}}
			\varphi_1(g) \varphi_2(g \epsilon_\infty) dg,
\]
where we use the measure $dg = \verts{y}^{-1} dx d^\times y dk$
for the Iwasawa decomposition
$g = z n(x) a(y) k$ with $k \in \K$,
the standard maximal compact subgroup of $\SL_2(\A)$.
The analogous local theory is developed in
\cite[Section 2]{zhang-rs}.

\subsection{The imaginary dihedral case}
Let $K/\Q$ be an imaginary quadratic field,
$\chi$ be a finite character of $G_K := \Gal(\overline{K}/K)$,
and $\rho_f = \Ind_{G_K}^{G_Q}(\chi)$.
Let $\xi$ be the antinorm $\chi^{1 - \Frob_\infty}$ and
let $c$ be the conductor of $\xi$.
We have automorphic
representations $\pi(\chi)$ and $\pi(\chi^{-1})$ of
imaginary dihedral weight $1$ forms,
where there is the automorphic avatar of optimal forms
of weight $(1, 1)$ as given in \cite[Definition 6.1]{zhang-hv},
\[
	\varphi^\opt\paren{g_1, g_2} \in \pi(\chi) \otimes \pi(\chi^{-1}),
\]
with an expression in terms of theta series (cf. Section \ref{subsec:theta}),
\[
	\varphi^\opt\paren{g, g\epsilon_\infty}
		= \sum_{\alpha \in \calo_c/\cald}
			\theta\Big(g, \chi, \Phi_\alpha^\opt\Big)
			\theta\paren{g\epsilon, \chi^{-1}, \Phi_{-\alpha}^\opt}.
\]

Recall from \cite[Equation 5.9]{zhang-hv}
that $T = K^\times \times K^\times / \Delta \Q^\times$
with norm map $\rmn(t_1, t_2) = \rmn(t_1 t_2^{-1})$
and,
\[
	\theta\paren{g, \one \otimes \xi, \Phi_\calo}
		= \int_{\sbrac{T^1}} \theta(g, t_0t, \Phi_\calo)
			\cdot (\one \otimes \xi)(t_0t) dt,
\]
where $t_0 \in T(\A)$ has the same norm as $g$.
By modifying the proof of \cite[Proposition 6.4]{zhang-hv}
to use the split quaternion algebra $B \iso M_2(\Q)$ with
an optimal embedding,
\[
	\calo_{c(\xi)} \hookrightarrow \calo \iso M_2(\Z),
\]
we have, 
\[
	\theta\paren{g, \iota_{T*}\paren{\one \otimes \xi}, \Phi_\calo}
		= \varphi^\opt\paren{g, g\epsilon_\infty},
\]
for the embedding $\iota_T: [T] \hookrightarrow [\GSO(V)]$.
The proof is the same except at the archimedean place,
where we have the decomposition
$B_\infty = K_\infty + K_\infty j_\infty$
with $j_\infty^2 = 1$ so $\Phi_\calo$ is the product of
characteristic function of $\wh{\calo}$ and the function,
\[
	\Phi_\infty(x + yj) := e^{-\pi\paren{\verts{x}^2 + \verts{y}^2}}.
\]

Let $\iota: [\Q^\times \bs K^\times] \hookrightarrow [\PGL_2]$
be the embedding
obtained from the optimal embedding $K \hookrightarrow M_2(\Q)$.
In terms of the theta lifting $\Theta$ defined by $\Phi_\calo$,
\begin{align*}
	\theta(g, \one \otimes \xi, \Phi_\calo)
		&= \Theta(\iota_{T*}(\one \otimes \xi)), \\
	\sbrac{\PSL_2(\Z) : \Gamma_0(N)} \cdot \calp_\rs\paren{\varphi^\opt}
		&= \pair{\iota_*(\one \otimes \xi)_T, \Theta^*(\one)}.
\end{align*}
Then by Theorem \ref{thm:theta-dual-eisenstein-2},
\begin{align*}
	\sbrac{\PSL_2(\Z) : \Gamma_0(N)} \cdot \calp_\rs\paren{\varphi^\opt}
		&= \pair{\iota_*(\one \otimes \xi)_T,
			\one \otimes \sce_0 + \sce_0 \otimes \one} \\
		&= \pair{\iota_{K*}(\one), \one} \cdot
			\pair{\iota_{K*}(\xi), \sce_0}
			+ \pair{\iota_{K*}(\one), \sce_0} \cdot
			\pair{\iota_{K*}(\xi), \one}.
\end{align*}
Using the identities (cf. \cite[Proof of Proposition 5.5]{dhrv}),
\begin{align*}
	\pair{\iota_{K*}(\one), \one} &= h\paren{\calo_c}, \\
	\pair{\iota_{K*}(\xi), \one} &= 0,
\end{align*}
we have,
\begin{equation}
	\label{eq:opt-iota}
	\sbrac{\PSL_2(\Z) : \Gamma_0(N)} \cdot \calp_\rs\paren{\varphi^\opt}
		= h\paren{\calo_c} \pair{\iota_{K*}(\xi), \sce_0}.
\end{equation}
 
Now, we translate everything in the terms of continuous 
functions on $\SL_2(\Z) \bs \calh$ so that we
can use the Kronecker limit formula for $\sce_0$:
\[
	\sce_0(x) = C - \frac{1}{12} \log \paren{\norm{\Delta}(x)},
\]
where $C$ is a constant and,
\[
	\norm{\Delta}(x) = \verts{y \Delta(\tau)},
\]
where $x$ is represented by $\tau \in \calh$.
As in Section \ref{sec:units},
$H_c$ is the ring class field of $\xi$,
$H_1$ is the Hilbert class
field of $K$, and $h_K = [H_1 : K]$
is the class number of $K$.
\begin{theorem}
	\label{thm:opt-rs}
	\[
		\sbrac{\PSL_2(\Z) : \Gamma_0(N)} \cdot \calp_\rs\paren{\varphi^\opt}
			= -\frac{\sbrac{H_c: H_1} w_K}{2} \, \Reg_\R\paren{u_\stark}.
	\]
\end{theorem}
 
\begin{proof}
	Write $\calo_K = \Z + \Z\tau$ with $\tau \in \calh$.
	Then for any positive integer $d$, we have a CM point $x_d$
	representing a CM elliptic curve $\C /(\Z + \Z d\tau)$ by $\calo_d$.
	Since $\xi$ is non-trivial, we have by Equation \ref{eq:opt-iota},
	\begin{equation}
		\label{eq:opt-delta}
		\sbrac{\PSL_2(\Z) : \Gamma_0(N)} \cdot \calp_\rs\paren{\varphi^\opt}
			= -\frac{1}{12} h\paren{\calo_c}
				\sum_{\sigma \in \Gal\paren{H_c/K}}
				\xi(\sigma) \log \paren{\norm{\Delta} \paren{x_c^\sigma}}.
	\end{equation}

	Consider the function,
  \[
		e_c(q) = \prod_{d \mid c} \Delta\paren{q^d}^{c\mu\paren{\frac{c}{d}}},
	\]
	where $\mu$ is the M\"obius function
	(cf. Proposition \ref{prop:stark}).
	Since $e_c$ has weight $0$, it is easy to see that,
	\[
		\verts{e_c(x_c)} = \prod_{d \mid c}
			\big(\norm{\Delta}\paren{x_d}\big)^{c\mu \paren{\frac{c}{d}}}.
	\]
	Thus,
	\[
		\sum_{\sigma \in \Gal\paren{H_c/K}} \xi(\sigma) \log \big(\verts{e_c \paren{x_c^\sigma}}\big)
			= \sum_{d \mid c} c \mu\paren{\frac{c}{d}} \sum_{\sigma \in \Gal\paren{H_c/K}}
				\xi(\sigma) \log \big(\norm{\Delta}\paren{x_d^\sigma}\big).
	\]
	For $d < c$, the second sum vanishes,
	as $x_d$ is defined over $H_d$ and
	$\xi$ is non-trivial on $\Gal(H_c/H_d)$.
	Therefore we only have the $d = c$ term,
	\[
		\sum_{\sigma \in \Gal\paren{H_c/K}} \xi(\sigma) \log \big(\verts{e_c \paren{x_c^\sigma}}\big)
			= c \sum_{\sigma \in \Gal\paren{H_c/K}}
				\xi(\sigma) \log \big(\norm{\Delta}\paren{x_c^\sigma}\big).
	\]
	Applying this to Equation \ref{eq:opt-delta} yields,
	\[
		\sbrac{\PSL_2(\Z) : \Gamma_0(N)} \cdot \calp_\rs\paren{\varphi^\opt}
			= -\frac{1}{12c} h\paren{\calo_c}
				\sum_{\sigma \in \Gal\paren{H_c/K}} \xi(\sigma)
				\log \big(\verts{e_c\paren{x_c^\sigma}}\big).
	\]
	The proof is completed by applying Proposition \ref{prop:stark} to the
	above and observing
	\[
		\Reg_\R\paren{u_\stark}
			= L'\big(\Ad(\rho_f), 0\big)
			= L(\eta, 0) \cdot L'(\xi, 0)
			= -\frac{h_K}{w_K} L'(\xi, 0).
	\]
\end{proof}


\section{Proof of the compatibility theorem}
\label{sec:proof-hvs}

To prove Theorem \ref{thm:hvs-imaginary}, we use the Rankin--Selberg
period from Section \ref{sec:rs-opt} and the
Harris--Venkatesh periods from \cite[Section 2]{zhang-hv},
which respectively extend the Petersson and Harris--Venkatesh norms
(with an extra factor of $[\PSL_2(\Z) : \Gamma_0(N)]^{-1}$)
to pairings on the dual cuspidal automorphic representations
$\pi_f$ and $\wt{\pi}_f = \pi_{f^*}$ of $\gl_2(\A)$ 
(respectively generated by $f$ and $f^*$
under the action of Hecke operators):
\begin{align*}
	\calp_\rs: \pi_f \otimes \wt{\pi}_f
		&\lra \R \otimes \Z\sbrac{\chi_{\Ad\paren{\rho_f}}, \frac{1}{6N}}, \\
	\calp_\hv: \pi_f \otimes \wt{\pi}_f
		&\lra \F_p^\times \otimes \Z\sbrac{\chi_{\Ad\paren{\rho_f}}, \frac{1}{6N}}.
\end{align*}
Following the multiplicity-one argument of
\cite[Theorem 6]{zhang-hv},
there is a reformulation of Conjecture \ref{conj:hvs}.
\begin{proposition}
	\label{prop:hvs-reform}
	Conjecture \ref{conj:hvs} is equivalent to the
	existence of an element $\varphi \in \pi_f \otimes \wt{\pi}_f$ with
	$\calp_\rs(\varphi) \neq 0$ and
	a unique element $u_\varphi \in \calu(\Ad(\rho_f)) \otimes \Q$
	such that,
	\[
		\calp_\rs(\varphi) = \Reg_\R \paren{u_\varphi},
	\]
	and such that for $p \gg 0$,
	\[
		\calp_\hv(\varphi) = \Reg_{\F_p^\times}\paren{u_\varphi}.
	\]
	Moreover, if Conjecture \ref{conj:hvs} is true then
	the above property holds for all
	$\varphi \in \pi_f \otimes \wt{\pi}_f$
	unramified away from $N$.
\end{proposition}

\begin{proof}
	By \cite[Proposition 3.1 and Section 4]{zhang-hv},
	there is a unique bilinear form $\calp$ on
	$\pi_f \otimes \wt{\pi}_f$ over $\Z[\chi_{\Ad\paren{\rho_f}}, 1/6N]$
	that is furthermore unique modulo $p$.
	The Harris--Venkatesh and
	Rankin--Selberg periods are therefore scalar multiples of
	$\calp$ over $\F_p^\times$ and $\R$
	respectively.
	Hence the ratios of $\calp_\hv$ and
	$\calp_\rs$ evaluated at any two test vectors
	$\varphi_1$ and $\varphi_2$ are equal
	(cf. \cite[Equation 4.1]{zhang-hv}),
	\begin{equation*}
			\alpha_{\varphi_1, \varphi_2} :=
			\sbrac{\calp_\hv\paren{\varphi_1} : \calp_\hv\paren{\varphi_2}}
			= \sbrac{\calp\paren{\varphi_1} : \calp\paren{\varphi_2}}
			= \sbrac{\calp_\rs\paren{\varphi_1} : \calp_\rs\paren{\varphi_2}}.
	\end{equation*}
	As in \cite[Section 4]{zhang-hv},
	if $\calp_\rs(\varphi_1), \calp_\rs(\varphi_2) \neq 0$
	then $\alpha_{\varphi_1, \varphi_2}$ is nonzero and rational.
	In particular, if there is a prime $p_0$
	such that for all primes $p \geq p_0$,
	\begin{align*}
		&\calp_\rs(\varphi_1)
			= \Reg_\R \paren{u_{\varphi_1}}, \qquad \hphantom{\alpha_{\varphi_1, \varphi_2} \cdot \,}
		\calp_\hv(\varphi_1)
			= \Reg_{\F_p^\times}\paren{u_{\varphi_1}}, \\
		\intertext{then for all primes $p \geq p_0$,}
		&\calp_\rs(\varphi_2)
			= \alpha_{\varphi_1, \varphi_2} \cdot
				\Reg_\R \paren{u_{\varphi_1}}, \qquad
		\calp_\hv(\varphi_2)
			= \alpha_{\varphi_1, \varphi_2} \cdot
				\Reg_{\F_p^\times}\paren{u_{\varphi_1}},
	\end{align*}
	Scaling $u_{\varphi_1}$ by $\alpha_{\varphi_1, \varphi_2}$
	gives the desired element
	$u_{\varphi_2}$ of $\calu(\Ad(\rho_f)) \otimes \Q$.
\end{proof}

For $f$ imaginary dihedral, the Stark conjecture for
$\Ad(\rho_f)$ is known
due to Stark \cite{stark-1980}
so we have a unique element $u_\stark$
of $\calu(\Ad(\rho_f)) \otimes \Q$
such that $L'(\Ad(\rho_f), 0) = \Reg_\R(u_\stark)$.
Recall from \cite[Section 6]{zhang-hv} that
for any imaginary dihedral modular form $f$
there is a two-variable optimal form $f^\opt$
with automorphic avatar $\varphi^\opt$ in
$\pi_f \otimes \wt{\pi}_f$ unramified outside of $N$.
Again with $w_K$ the number of roots of unity of
the imaginary quadratic number field $K$,
$H_c$ the ring class field associated to
$\xi := \chi^{1 - \Frob_\infty}$,
and
\[
	m(\xi) := 
		\begin{cases}
			v & \text{if $\Im(\xi)$ is a $v$-group,} \\
			1 & \text{otherwise,}
		\end{cases}
\]
take $u_{\varphi^\opt} = - \frac{[H_c : H_1] w_K}{2 \, [\PSL_2(\Z) : \Gamma_0(N)]} \otimes u_\stark$.
This is a suitable element of $\calu(\Ad(\rho_f)) \otimes \Q$
that evaluates at $x_\infty$
to $-\frac{[H_c : K]}{12m(\xi) \, [\PSL_2(\Z) : \Gamma_0(N)]} \otimes u_\xi$
by Proposition \ref{prop:stark-dihedral}.
We conclude the proof of Theorem \ref{thm:hvs-imaginary}
by using our results to show that $\varphi^\opt$
and $u_{\varphi^\opt}$ satisfy the conditions
of Proposition \ref{prop:hvs-reform}.

\begin{lemma}
	\label{lem:hvs-CM}
	Let $f$ be an imaginary dihedral form of weight $1$.
	For its associated optimal form $\varphi^\opt$,
	\begin{enumerate}[(a)]
		\item $\calp_\rs(\varphi^\opt) = \Reg_{\R}(u_{\varphi^\opt})$;
		\item $\calp_\hv(\varphi^\opt) = \Reg_{\F_p^\times}(u_{\varphi^\opt})$
			for all primes $p \nmid 6N$.
	\end{enumerate}
\end{lemma}

\begin{proof}
	\,

	(a): By Theorem \ref{thm:opt-rs}:
	\[
		\calp_\rs\paren{\varphi^\opt}
			= -\frac{\sbrac{H_c : H_1} w_K}{2 \, [\PSL_2(\Z) : \Gamma_0(N)]} \Reg_\R\paren{u_\stark}
			= \Reg_\R\paren{u_{\varphi^\opt}}.
	\]

	(b): By Corollary \ref{cor:stark-opt}, for all primes $p \nmid 6N$:
	\[
		\calp_\hv \paren{f^\opt}
			= -\frac{\sbrac{H_c : H_1} w_K}{2 \, [\PSL_2(\Z) : \Gamma_0(N)]} \Reg_{\F_p^\times} \paren{u_{\stark}}
			= \Reg_{\F_p^\times}\paren{u_{\varphi^\opt}}.
	\]
\end{proof}


\bibliography{bibliography}{}
\bibliographystyle{amsalpha}

\end{document}